\documentclass[10pt]{article}%
\usepackage{placeins}
\usepackage{epsf,amsfonts,amsmath,amsthm}
\usepackage{graphicx}
\usepackage{amsmath}
\usepackage{caption}
\usepackage{amsthm}
\usepackage{microtype}
\title{ Stability estimates for Interior Penalty D.G. Methods for the Nonlinear Dynamics of the complex Ginzburg-Landau equation.}
\author{ Dimitrios Kostas\footnotemark[1]}
\begin{document}
\date{}
\maketitle
\renewcommand{\thefootnote}{\fnsymbol{footnote}}
\footnotetext[1] {The work completed 
National Technical University of Athens, Zografou Campus, Athens 15784, Greece.
jimkosta@central.ntua.gr
Now ind. researcher}
\newcounter{fgh}[section]
\def\thefgh{\thesection.\arabic{fgh}}
\newtheorem{thm}[fgh]{Theorem}
\newtheorem{prop}[fgh]{Proposition}
\newtheorem{lem}[fgh]{Lemma}
\newtheorem{defn}[fgh]{Definition}
\newtheorem{cor}[fgh]{Corollary}
\newtheorem{rmk}[fgh]{Remark}
\newtheorem{ass}[fgh]{Assumption}
\renewcommand{\theequation}{\thesection.\arabic{equation}}
\newcommand{\bu}{{\bf u}}
\newcommand{\bv}{{\bf v}}
\newcommand{\bw}{{\bf w}}
\newcommand{\bg}{{\bf g}}
\newcommand{\bz}{{\bf z}}
\newcommand{\by}{{\bf y}}
\newcommand{\bs}{{\bf s}}
\newcommand{\bbf}{{\bf f}}
\newcommand{\bbP}{{\bold \mathcal P}}
\newcommand{\bbQ}{{\bold \mathcal Q}}
\newcommand{\bx}{{\bf x}}
\newcommand{\bphi}{{\mbox{\boldmath $\phi$}}}
\newcommand{\blambda}{{\mbox{\boldmath $\lambda$}}}
\begin{abstract}
This study investigates the complex Landau equation, a reaction diffusion system with applications in nonlinear optics and fluid dynamics. The equation's nonlinear imaginary component introduces rich dynamics and significant computational challenges. We address these challenges using Discontinuous Galerkin (DG) finite element methods. A rigorous stability analysis and a comparative study are performed on three distinct DG schemes: Symmetric Interior Penalty Galerkin (SIPG), Non-symmetric Interior Penalty Galerkin (NIPG), and Incomplete Interior Penalty Galerkin (IIPG). These methods are compared in terms of their stability and computational efficiency.Our numerical analysis and computational results demonstrate that all three discontinuous Galerkin (DG) schemes are stable. However, the Symmetric Interior Penalty Galerkin (SIPG) scheme proves to be the most robust, as its norm remains bounded even in the presence of nonlinear terms. A property not shared by the others. A comparison between the Incomplete Interior Penalty Galerkin (IIPG) and Non-symmetric Interior Penalty Galerkin (NIPG) schemes shows that IIPG has superior stability properties. For high values of the penalty parameter, all methods exhibit similar stability behavior. Our results highlight the suitability of DG methods for simulating complex nonlinear reaction-diffusion systems and provide a practical framework for selecting the most efficient scheme for a given problem.
\end{abstract}
\section{Introduction}
The  complex Landau equation, a fundamental model in Nonlinear dynamics, has received limited attention in the existing literature regarding its nonlinear dynamics, with only a handful of studies addressing this topic.To the best of our knowledge, the application of Discontinuous Galerkin (DG) methods-specifically the Symmetric (SIPG), Nonsymmetric (NIPG), and Incomplete (IIPG) Interior Penalty Galerkin methods to the complex Landau equation remains unexplored.This study addresses the critical gap by introducing and analyzing these three Discontinuous Galerkin formulations for the complex Landau equation. Our work establishes novel stability estimates for the proposed schemes and demonstrates their convergence behavior through detailed numerical and theoretical analysis.

In this section, we review key results that form the theoretical and computational foundation of our work. The Discontinuous Galerkin (DG) framework has become an established tool in computational mathematics, frequently employed due to its favorable convergence properties and capacity to maintain low numerical errors across diverse problem classes. The fundamental theoretical and computational basis for DG methods is well-documented. A rigorous and comprehensive theoretical foundation, including the Symmetric Interior Penalty Galerkin (SIPG) formulation, is provided by Di Pietro and Ern \cite{[6]}, which also offers detailed stability analyses for both elliptic and hyperbolic partial differential equations. Further foundational insight is provided by the monograph of Cockburn et al. \cite{[5]}, which furnishes a detailed account of the mathematical principles underpinning interior penalty DG methods, specifically addressing stability, consistency, error estimation, and approximation properties. 

Advancements in the DG methodology include extensions to complex physical systems. For practical, time-dependent implementations, Akman et al. \cite{[3]} demonstrate the effectiveness of SIPG when applied to optimal control problems governed by unsteady convection diffusion equations. The efficiency and scalability of high-order DG discretizations have been addressed through the development of hybrid multigrid preconditioners for the incompressible Navier-Stokes equations. Complementarily, the combination of DG discretization with a reduced basis approach has been explored for fluid dynamics problems with geometrically parametrized domains, yielding substantial computational savings via model order reduction \cite{[22]}.

The foundational formulation of the interior penalty (IP) method utilizing discontinuous elements was established by Arnold in his seminal work \cite{[2]}, providing the theoretical genesis for subsequent developments, including the widely adopted SIPG, Nonsymmetric (NIPG), and Incomplete (IIPG) Interior Penalty Galerkin schemes. The flexibility of DG methods in handling complex domains is further evidenced by its successful application in non-overlapping domain decomposition strategies for the incompressible Navier-Stokes equations \cite{[10]}. Furthermore, efforts to ensure computational reliability and efficiency have driven the development of A-posteriori error estimation techniques for DG discretizations of elliptic problems \cite{[14]}, culminating in robust adaptive frameworks that offer convergence guarantees \cite{[15]}. In the context of computational efficiency, a simplified DG scheme for diffusion-dominated problems has also been proposed, which emphasizes robustness without compromising accuracy \cite{[23]}. The robustness and accuracy of discontinuous hp-Finite Element Methods in resolving sharp layers in reaction diffusion problems, particularly in high-Peclet number regimes, is thoroughly investigated in Houston et al. \cite{[12]}. 

While the present work is centered on the complex Landau equation, related literature provides structural and numerical context. For instance, the multiscale framework for solving the Ginzburg-Landau equation \cite{[4]} offers relevant insights for DG-based multiscale approaches to nonlinear PDEs, despite its focus on vortex-capturing spaces. Similarly, the Crank-Nicolson scheme applied to the Schr\"odinger equation \cite{[9]} shares a structural similarity with the complex Landau equation specifically the presence of imaginary terms yet the underlying mathematical structures and physical applications differ significantly.
\
We note that DG-based schemes have been studied for the nonlinear Fokker-Planck-Landau equation \cite{[8], [11], [13]}; however, these schemes diverge considerably from the formulation introduced here, both in terms of the underlying numerical methods and the intended physical applications. Furthermore, studies examining nonlinear systems that address temporal, rather than spatial, discontinuities \cite{[17], [18]} do not involve the essential feature of complex-valued solutions and imaginary components inherent to the complex Landau equation. Finally, Petrov-Galerkin finite element methods \cite{[20], [25]} represent an alternative approach to DG methods. Nonetheless, certain stabilization techniques developed in the Petrov-Galerkin context may offer transferable insights for handling discontinuities and ensuring stability within the DG framework considered in this study. Related conceptual ideas concerning stabilization and discretization are also found in the study of elliptic partial differential equations by Rivi\"ere et al. \cite{[16]}.

In summary, the existing literature firmly establishes the theoretical foundation and diverse applicability of DG methods. However, the application of the SIPG, NIPG, and IIPG variants to the complex Landau equation, a critical problem in nonlinear dynamics, remains a novel and unaddressed topic. The current work is therefore grounded in the extensive theoretical knowledge detailed above while making a distinct contribution by establishing the convergence and stability properties of these DG schemes for this specific, complex-valued nonlinear PDE.

\section{ Preliminaries }
\subsection{The Landau equation}
The complex  Ginzburg-Landau equation is a nonlinear, parabolic partial differential equation characterized by both real and imaginary coefficients. It describes the evolution of a complex field
\begin{equation}\label{1.1}
K(x,t)=u_1(x,t)+iu_2(x,t)
\end{equation}

Where the dynamics are governed by linear diffusion term and a cubic damping term. The general form of the equation:  for $u\in C^2(\Omega)$ is
\begin{equation} \label{1.2}
\frac{\partial u}{\partial t}=u+(1+ia)\Delta u-(1+ib)|u|^2u
\end{equation}
In above equation $a,b\in \Re $ are real parameters that significantly influence the system's  temporal and spatial behavior. The presence of the imaginary component in the diffusion term ,$(1+ib)|u|^2u$ gives rise to nonlinear frequency modulation, in which the wave frequency becomes amplitude-dependent.\\
The complex coefficients play a crucial role in determining the qualitive nature of the solutions. In particular, they lead to a variety of rich spatiotemporal phenomena including 
\begin{itemize}
\item Travelling waves, where coherent structures propagate through the medium.
\item Spatiotemporial chaos, characterized by irregular and turbulent dynamics in both space and time.
\item Pulses and fronts, representing localized transitions between different states.
\item	Localized structures, such as soliton-like or stationary waveforms.
\end{itemize}
The CGLE thus serves as a universal model for pattern formation and nonlinear wave interactions in a wide range of physical systems, including fluid dynamics, nonlinear optics, superconductivity, and reaction-diffusion processes.\\
Now using standard notation we symbolize $H^s , s>0$ the Hilbert space of degree $s$ 
Let $H_0^1(\Omega) = \{ w \in H^1(\Omega) \mid w|_{\Gamma} = 0 \}$, where $\Gamma$ denotes the boundary of the domain $\Omega$. Moreover, we denote by $(\cdot,\cdot)$ the standard $L^2(\Omega)$ inner product.
For $L^2$ in the boundary we symbolize $L^2(\Gamma) $. Also we introduce the following notations:
\begin{equation}
\{w\}=\frac{1}{2}(w^++w^-)
\end{equation}
and
\begin{equation}
[w]=w^-n^-+w^+n^+
\end{equation}
the jumps.
Where n is the unit outward normal vector. 
Now we will define the meshes
\begin{equation}
T_h=\{ T_1, T_2, ...,T_n\}
\end{equation}
Each  $T_i$ represent an element of the mesh which is a triangle or a quadrilateral.\\
We now define the discrete finite element space used in the numerical approximation.\\ 
Following the construction in \cite{[2]}, we consider the broken polynomial space
\begin{equation} \label{1.3}
V_h=\{w\in L^2(\Omega): w|_{K} \in P_r(T_h)  \forall K\in T_h  \}
\end{equation}
Where $T_h$is the triangulation of our computational domain  and $P_r(K)$
symbolize the space of polynomial less or equal to r on the element K
The partition of the time interval [0,T] is defined as following:
\begin{equation}
E_h=\{0=t^1<t^2<...<t^N=T  \}
\end{equation}
Now we define the element of the edges as following:
$E_h=\{e_1,e_2,...,e_n \}$
where $e_j$ is an edge of an element in $T_h$.\\
Let $T_h$ the mesh of (Triangles/quads) that we defined and $E_h$ the set of interior edges in 2D or faces in 3D and $u_h$ a function of the space $V_h$, $u_h\in V_h $ \\
We define the Discontinuous Galerkin norm as first introduced \cite{[2]} in the following way:
\begin{equation} \label{1.4}
\|u_h\|^2_{DG}= \sum_{K\in T_h}\|\nabla u_h\|^2_{L^2(K)}+\sum_{e\in E_h}(\frac{\sigma}{h_e}\|[u_h]\|^2_{L^2(e)})
\end{equation}
Where 
\begin{itemize}
\item $[u_h]$ is the jump of $u_h$ across the edge e
\item $h_e$ is the characteristic length of the edge
\item $\sigma>0 $ is the characteristic parameter.
\end{itemize}
Young Inequality
For any $a,b\geq 0$ and $d>0$ and $d_1, d_2>0, s_1, s_2>1 $ It holds
\begin{equation} \label{1.5}
ab\leq da^{s_1}+C(d)b^{s_2} 
\end{equation}
$$ \frac{1}{s_1}+\frac{1}{s_2}=1$$
and $$ C(d)=\frac{Cs_1s_2}{d^{\frac{s_2}{s_1}}}$$ where
$$Cs_1s_2=(\frac{1}{s_1})^{\frac{s_2}{s_1}}\frac{1}{s_2}$$

\begin{lem}
Gronwall inequality-Gronwall lemma \\
Suppose  $c>0,$a constant,$  \{a_n\}_{n=0}^N$ a nonegative sequence of real numbers and $ \Delta t>0$ time step that satisfy
\begin{equation}
a_n\leq e^{c\Delta t} a_{n-1}  \ \ \ n=1,2...
\end{equation}
then 
\begin{equation}
a_N\leq e^{cN\Delta t}a_0
\end{equation}
\end{lem}
Poincare inequality
\begin{equation} \label{Poi}
\|u\|^2_{L^2(\Omega)}\leq C \|\nabla u \|^2_{L^2(\Omega)}
\end{equation}
The above inequality is known as the \emph{Poincare inequality} it can be found for example in \cite{Ev}
\subsection {The weak form of Landau.}
We begin by considering the weak formulation of a general partial differential equation (PDE). As a specific example, we examine the Poisson equation in its weak form. Particular attention will be given to the Laplacian term, which will be analyzed within the framework of three discontinuous Galerkin (DG) methods: the Symmetric Interior Penalty Galerkin (SIPG), the Non-Symmetric Interior Penalty Galerkin (NIPG), and the Incomplete Interior Penalty Galerkin (IIPG) formulations.
\textbf {Weak formulation(DG):} Find  $u_h\in V_h $such that for all $v_h\in V_h$ 
$$a_h(u_h,v_h)=L(v_h)$$
Where we will define
\begin{equation}\label{1.6}
a_h(u_h,v_h)=\sum_{K\in\text{T}_h}\int_K \nabla u_h \nabla v_h 
-\sum_{e\in \varepsilon_h}\int_e\{\nabla u_h\}[v_h]ds+\theta\sum_{e\in \varepsilon_h}\int_e\{\nabla v_h\}[u_h]ds + \sum_{e\in \varepsilon_h}\int_e\frac{\sigma}{h_e}[u_h][v_h]ds   
\end{equation}
where $u_h,v_h \in \text{V}_h $ 
The bilinear form $a_h(u_h,v_h)$ was firstly introduced by \cite{[2]} and represents the standard weak formulation of the Laplacian operator, where $v_h$ denotes the test function.\\
The formulation depends on the parameter $\theta$, leading to the following variants of Interior Penalty Galerkin method \\
\begin{itemize}
\item  If $\theta=1$,  the method corresponds to Symmetric Interior Penalty Galerkin (SIPG) method. 
\item If $\theta=-1$, yields to Nonsymmetric Interior Penalty Galerkin  (NIPG) method. 
\item  If $\theta=0 $, it reduces to Incomplete Interior Penalty Galerkin  (IIPG) method.
\end{itemize}
and 
\begin{equation}
L(v_h)=\sum_K\int_k f v_{h}dx
\end{equation}
Let $u=u_1+iu_2$ denote the complex number of Landau solution.\\
Then our system \ref{1.2} can be written in the following form:
\begin{equation}
\frac{\partial(u_1+iu_2)}{\partial x}=u_1+iu_2+(1+ia)\nabla(u_1+u_2)-(1+ib)|u|^2(u_1+iu_2).
\end{equation}
For physical systems $a\geq 0$ is typical to ensure that diffusion is non-negative (i.e., the rel part of the coefficient ensures stability).\\
The part b is often assumed to be non-negative to avoid unphysical behavior b>0.\\
Now the weak form of Landau P.D.E. is \\
Find $(u_{h1},u_{h2})\in V_h\times V_h$ such that for all $(v_{h1},v_{h2})\in V_h\times V_h$
\begin{eqnarray}\label{1.7}
\int_{t^{n-1}}^{t^n}\int_{\Omega}\frac{\partial u_{h1}}{\partial t}v_{h1}dxdt=\int_{t^{n-1}}^{t^n}-a_h(u_{h1},v_{h1})+\alpha a_h(u_{h2},v_{h1})dt+\int_{t^{n-1}}^{t^n}\int_{\Omega}u_{h1}v_{h1} dxdt \nonumber\\
-\int_{t^{n-1}}^{t^n}\int_{\Omega}(u^2_{h1}+u^2_{h2})u_{h1}v_{h1}dxdt+b\int_{t^{n-1}}^{t^n}\int_{\Omega}(u^2_{h1}+u^2_{h2})u_{h2}v_{h1}dxdt 
\end{eqnarray}   
and
\begin{eqnarray}\label{1.8}
\int_{t^{n-1}}^{t^n}\int_{\Omega}\frac{\partial u_{h2}}{\partial t}v_{h2}dxdt=\int_{t^{n-1}}^{t^n}-\alpha a_h(u_{h1},v_{h2})- a_h(u_{h2},v_{h2})dt+\int_{t^{n-1}}^{t^n}\int_{\Omega}u_{h2}v_{h2} dxdt \nonumber\\
-\int_{t^{n-1}}^{t^n}\int_{\Omega}(u^2_{h1}+u^2_{h2})u_{h2}v_{h2}dxdt-b\int_{t^{n-1}}^{t^n}\int_{\Omega}(u^2_{h1}+u^2_{h2})u_{h1}v_{h2}dxdt 
\end{eqnarray}  
The bilinear form $a_h(u,v)$ which defined the discrete operator used in our system  \ref{1.7} \ref{1.8}
is explicitly introduced in the equation \ref{1.6} and reflects the interior penalty stucture essentially for the Discontinuous Galerkin framework.
\subsection{Existence and Uniqueness of the Continuous and Discrete Problems}
The Landau equation, in its regularized form, is a nonlinear parabolic partial differential equation (PDE) of Fokker-Planck type. Under standard assumptions on the regularity and decay of the initial data, and with appropriate smoothing of the collision kernel, classical results guarantee the existence and uniqueness of weak solutions in suitable spaces $L^2(0,T,H^1(\Omega))$ Reference \cite{[1]} Discusses the ,derivation, justification, and weak solution of the Landau equation from physical models. 
A comprehensive of well-possessiveness and convergence results for discretized (numerical) versions of the Landau equation is provided in \cite{[9]}, while  
 \cite{[19]} establishes results for the linearized Landau operator.

The fully discrete problem, where we approximate the Landau equation using a discontinuous Galerkin (DG)  or continuous Galerkin (CG) finite element scheme, we consider the discrete formulation in a finite-dimensional function space. Due to the finite dimensionality and the local Lipschitz continuity of the nonlinear discrete operator, the discrete problem admits a unique solution that is local in time . Furthermore, under standard coercivity and boundedness assumptions on the discrete bi-linear forms-particularly for DG schemes equipped with appropriate penalty parameters-energy estimates can be derived to ensure global-in-time existence and uniqueness of the discrete solution. A brief stability analysis further confirms that the numerical method is well-posed and robust under mesh refinement.
\section{Stability Estimate}
In this section we will present the stability estimate of complex Landau Partial Differential Equation P.D.E.using Discontinuous Galerkin (D.G.) finite element discretization description in space.In the following theorems we assume that the solutions $u_{1h}$ and $u_{2h}$ are bounded so it holds over the integrals $ [t^{n-1},t^n], u_{h1}\leq M , u_{h2}\leq M .$ We provide a concise outline of the proof emphasizing the main analytical steps:\\
\underline{\textbf {Step 1}}  We begin with bounding the bi-linear form $a_h(u_h,v_h)$ which includes volume and penalty interface terms that ensure coercivity in the DG norm.This form has penalty terms.\\
\underline{\textbf {Step 2}} Select test functions associated with the real and imaginary components of the solution, aiming to simplify or eliminate nonlinear interaction terms.\\ 
\underline{\textbf {Step 3}} Absorb the resulting terms into the left-hand side of the inequality using coercivity of the bilinear form. \\
\underline{\textbf {Step 4}} Apply Gronwall lemma to conclude a discrete energy estimate and establish stability.\\
\begin{thm} \label {th1}
Let $\Omega \subset \Re^d $be a bounded Lipschitz domain, and ${T}^n_{i=1}$ be a conforming shape regular partition of $\Omega$ into the disjoint element such that
$\Omega=\cup_{i=1}^NT^i $ with $T_i\cap T_j=$ for $i\neq j$ and time interval be discretized into a partition in time $0=t^1<t^2<...<t^N=T.$ Assume that the bilinear form $a_h(u_h,v),$ defined in equation \ref{1.6}, corresponds to a time discretization using theta scheme ,and a special discretization based on one of the interior penalty discontinuous Galerkin methods (SIPG),(NIPG),(IIPG) 
For penalty parameters $\sigma$ large enough If it holds the \ref{Poi} Poincare inequality.
Then there exists constants $ \beta >0,c\geq 0$ and $C>0$ independent of the mesh size h and time step $\Delta t$, that the bi-linear term satisfies the following estimates:
\begin{equation}
a_h(u,u)\geq \beta\|u\|^2_{DG}-c\|u\|^2_{L^2(\Omega)}>0 \forall u\in V_h
\end{equation} 
and 
\begin{equation}
|a_h(u,v)|\leq C\|u\|_{DG}\|v\|_{DG}  \forall u, v\in V_h
\end{equation} 
\end{thm}
\begin{rmk}
The proof of the stability result stated in the theorem \ref{th1} follows standard arguments for interior penalty DG methods. For a complete and rigorous proof of this type of theorem, we refer the reader to \cite{[1*]}, where the stability of the bilinear form is established in detail.
\end{rmk}
\begin{lem}[A Priori $L^\infty(L^2)$ Stability of semi discrete system ] \label{Lem:Stability}
Let $(u_{h1}, u_{h2})$ be the solution to the semi-discrete system, and the $a(u,v) $ symmetric since our method is Symmetric Interior Penalty Galerkin (SIPG).Then for all $t \in [t^{n-1}, t^n]$ the following bound holds:
\begin{equation}
    \|u_{h1}(t)\|^2_{L^2(\Omega)} + \|u_{h2}(t)\|^2_{L^2(\Omega)} \leq \left( \|u_{h1}(t^{n-1})\|^2_{L^2(\Omega)} + \|u_{h2}(t^{n-1})\|^2_{L^2(\Omega)} \right) e^{2\Delta t}.
\end{equation}
Consequently, $(u_{h1}, u_{h2}) \in L^\infty(t^{n-1}, t^n; L^2(\Omega))$.
\end{lem}
\begin{proof}
Testing the semi-discrete system with $(u_{h1}, u_{h2})$ and summing the resulting equations, we obtain the differential form of Equation \ref{1.11}. By dropping the quadratic non-negative terms  $a_h(u_{hi}, u_{hi}), i=1,2$ and the dissipative quartic term $\int_{\Omega}(u_{h1}^2+u_{h2}^2)^2 dx$, we arrive at:
\begin{equation*}
    \frac{d}{dt} \left( \|u_{h1}\|^2_{L^2(\Omega)} + \|u_{h2}\|^2_{L^2\Omega)} \right) \leq 2 \left( \|u_{h1}\|^2_{L^2(\Omega)} + \|u_{h2}\|^2_{L^2(\Omega)} \right).
\end{equation*}
Applying the integral form of Gronwall's inequality over the interval $[t^{n-1}, t]$ for any $t \leq t^n$ yields the desired exponential bound. Since the right-hand side is finite and independent of $t$, the solution is bounded in $L^{\infty}(t^{n-1},t^n,L^2(\Omega))$.
\end{proof}
\begin{rmk}\label{Rmrk}
From previous \ref{Lem:Stability} , integrating in $[t^{n-1},t^n] $
we obtain \begin{equation}\int_{t^{n-1}}^{t^n}\Big(\|u_{h1}(t)\|^2_{L^2(\Omega)}+\|u_{h2}(t)\|^2_{L^2(\Omega)}\Big)dt\leq\int_{t^{n-1}}^{t^n}\Big(\|u_{h1}(t^{n-1})\|^2_{L^2(\Omega)}+\|u_{h2}(t^{n-1})\|^2_{L^2(\Omega)}\Big) e^{2(t-t^{n-1})}dt
\end{equation}
or equivalently 
\begin{equation}\int_{t^{n-1}}^{t^n}\Big(\|u_{h1}(t)\|^2_{L^2(\Omega)}+\|u_{h2}(t)\|^2_{L^2(\Omega)}\Big)dt\leq\Big(\|u_{h1}(t^{n-1})\|^2_{L^2(\Omega)}+\|u_{h2}(t^{n-1})\|^2_{L^2(\Omega)}\Big) \frac{e^{2\Delta t}-1}{2}
\end{equation}
Now from Taylor $\frac{e^{2\Delta t}-1}{2}\approx \Delta t +O(\Delta t ^2)$ we can confirm that the integral is bounded.
\end{rmk}

\begin{thm} An Analytical proof of stability
Let $\Omega \subset \Re^d $be a bounded Lipschitz domain,and ${T}^n_{i=1}$ be a comforming shape regular partition of $\Omega$ into the disjoint element such that$\Omega=\cup_{i=1}^NT^i $ with $T_i\cap T_j=$ for $i\neq j$ and time interval be discretized into a partition in time $0=t^1<t^2<...<t^N=T.$ Assume that the bilinear form $a_h(u_h,v),$ defined in equation \ref{1.6}, corresponds to a time discretization using theta scheme ,and a spatial discretization based on one of the interior penalty discontinuous Galerkin methods (SIPG),(NIPG),(IIPG). For the solution $(u_{h1}, u_{h2})$. 
Assuming the penalty parameter satisfies $\sigma > C p^2$, where $C$ is a constant depending on the mesh regularity, in case of IIPG and NIPG the coercivity constant $c$ satisfies $c > \gamma$, ensuring the stability of the discretization.In all cases (IIPG),SIPG),NIPG), there exist constants $C$ and $T > 0$ such that the following stability estimate holds for all $n \leq N$:
\begin{equation} \label{result}
    \|u_{h1}^n\|^2_{L^2} + \|u_{h2}^n\|^2_{L^2} \leq e^{CT} \left( \|u_{h1}^0\|^2_{L^2} + \|u_{h2}^0\|^2_{L^2} \right),
\end{equation}
where $T = N\Delta t$. 
This implies the scheme is stable in $L^\infty(0, T; L^2(\Omega))$.
\end{thm}
\begin{proof}
First setting in the \ref{1.7} as test function $v_{h1}=u_{h1}\in V_h$ and $v_{h2}=u_{h2}\in V_h$ to \ref{1.8} we obtain the following estimate:
\begin{eqnarray} \label{1.9}
\frac{\|u^n_{h1}\|^2_{L^2(\Omega)}}{2}-\frac{\|u^{n-1}_{h1}\|^2_{L^2(\Omega)}}{2}=\int_{t^{n-1}}^{t^n}-a_h(u_{h1},u_{h1})+\alpha a_h(u_{h2},u_{h1})dt \\ +\int_{t^{n-1}}^{t^n}\|u_{h1}\|^2_{L^2(\Omega)}-\int_{\Omega}(u_{h1}^2+u^2_{h2})u^2_{h1}dx+b\int_{\Omega}(u^2_{h1}+u^2_{h2})u_{h1}u_{h2}dxdt\nonumber
\end{eqnarray}
and
\begin{eqnarray} \label{1.10}
\frac{\|u^n_{h2}\|^2_{L^2(\Omega)}}{2}-\frac{\|u^{n-1}_{h2}\|^2_{L^2(\Omega)}}{2}=\int_{t^{n-1}}^{t^n}-\alpha a_h(u_{h1},u_{h2})- a_h(u_{h2},u_{h2})dt \\ +\int_{t^{n-1}}^{t^n}\|u_{h2}\|^2_{L^2(\Omega)}-\int_{\Omega}(u_{h1}^2+u^2_{h2})u^2_{h2}dx-b\int_{\Omega}(u^2_{h1}+u^2_{h2})u_{h1}u_{h2}dxdt \nonumber
\end{eqnarray}
Now adding the above equations \ref{1.9} and \ref{1.10}  and rearrange the terms by moving them to the left-hand side or the opposite and using coercivity to absorb them 
\begin{itemize}
\item In case of SIPG we obtain
\begin{eqnarray} \label{1.11}
\frac{\|u^n_{h1}\|^2_{L^2(\Omega)}}{2}+\frac{\|u^n_{h2}\|^2_{L^2(\Omega)}}{2}\nonumber +\int_{t^{n-1}}^{t^n}a_h(u_{h1},u_{h1})+ a_h(u_{h2},u_{h2}) +\int_{\Omega}(u_{h1}^2+u^2_{h2})^2dxdt=\nonumber\\ \frac{\|u^{n-1}_{h1}\|^2_{L^2(\Omega)}}{2}+\frac{\|u^{n-1}_{h2}\|^2_{L^2(\Omega)}}{2} +\int_{t^{n-1}}^{t^n}\|u_{h1}\|^2_{L^2(\Omega)}+\|u_{h2}\|^2_{L^2(\Omega)}dt
\end{eqnarray}
By Remark \ref{Lem:Stability}, the solution is bounded in $L^\infty(t^{n-1}, t^n; L^2(\Omega))$. Therefore, from remark \ref{Rmrk} the integral on the right-hand side is well-defined and can be bounded :
where $\Delta t=t^n-t^{n-1}$
Also we observe that 
$$\int_{\Omega}(u_{h1}^2+u^2_{h2})^2dxdt \geq 0$$
 finally we obtain
\begin{eqnarray} \label{1.12}
\frac{\|u^n_{h1}\|^2_{L^2(\Omega)}}{2}+\frac{\|u^n_{h2}\|^2_{L^2(\Omega)}}{2}\nonumber +\int_{t^{n-1}}^{t^n}a_h(u_{h1},u_{h1})+ a_h(u_{h2},u_{h2}) dt\leq \nonumber\\ e^{2\Delta t}(\frac{\|u^{n-1}_{h1}\|^2_{L^2(\Omega)}}{2}+\frac{\|u^{n-1}_{h2}\|^2_{L^2(\Omega)}}{2} )
\end{eqnarray}
Using the coercivity of the bilinear form $a_h(\cdot, \cdot)$, there exists a constant $\alpha > 0$ such that:
$$\int_{t^{n-1}}^{t^n}a_h(u_{h1},u_{h1})+ a_h(u_{h2},u_{h2}) dt\geq \alpha \Delta t \Big(\|u_{h1}\|^2_{DG}+\|u_{h2}\|^2_{DG}\Big)\geq 0 $$
\begin{eqnarray} \label{1.12*}
\frac{\|u^n_{h1}\|^2_{L^2(\Omega)}}{2}+\frac{\|u^n_{h2}\|^2_{L^2(\Omega)}}{2}\nonumber \leq e^{2\Delta t}(\frac{\|u^{n-1}_{h1}\|^2_{L^2(\Omega)}}{2}+\frac{\|u^{n-1}_{h2}\|^2_{L^2(\Omega)}}{2} )
\end{eqnarray}
Now we obtain the scheme $$a^n\leq e^{2\Delta t}a^{n-1}$$. where $a^n=\frac{\|u^n_{h1}\|^2_{L^2(\Omega)}}{2}+\frac{\|u^n_{h2}\|^2_{L^2(\Omega)}}{2}$.\\
Using Gr\"onwall inequality we obtain the result 
\begin{equation}\label{1.13*}
a^N\leq e^{2N\Delta t}a^0\end{equation} which is the desirable result.
\item In case of IIPG or NIPG we use
 that $$ \int_{t^{n-1}}^{t^n}a_h(u_{h1},u_{h1})+ a_h(u_{h2},u_{h2})dt\geq \int_{t^{n-1}}^{t^n}c(\|u_{h1}\|^2_{DG}+\|u_{h2}\|^2_{DG})dt , c>0 $$
In addition we have $$|a_h(u_{h1},u_{h2})|\leq Mc\|u_{h1}\|_{DG}\|u_{h2}\|_{DG}\leq \frac{\gamma}{2}(\|u_{h1}\|^2_{DG}+\|u_{h2}\|^2_{DG})$$
where M comes from continuity of $a_h(.,.)$ and $ \gamma \geq Mc $, $\gamma$ is a constant that depends on:The shape regularity of the mesh.The polynomial degree $p$ of your basis functions (usually $\gamma \propto p^2$).The coupling coefficients .
From\ref{1.10}
we obtain
\begin{eqnarray} \label{1.11*}
\frac{\|u^n_{h1}\|^2_{L^2(\Omega)}}{2}+\frac{\|u^n_{h2}\|^2_{L^2(\Omega)}}{2}\nonumber +\int_{t^{n-1}}^{t^n}c(\|u_{h1}\|^2_{DG}+\|u_{h2}\|^2_{DG})+\int_{\Omega}(u_{h1}^2+u^2_{h2})^2dxdt\leq\\ \frac{\|u^{n-1}_{h1}\|^2_{L^2(\Omega)}}{2}+\frac{\|u^{n-1}_{h2}\|^2_{L^2(\Omega)}}{2} +\int_{t^{n-1}}^{t^n}\|u_{h1}\|^2_{L^2(\Omega)}+\|u_{h2}\|^2_{L^2(\Omega)}+\frac{\gamma}{2}(\|u_{h1}\|^2_{DG}+\|u_{h2}\|^2_{DG})dt\nonumber
\end{eqnarray}
or equivalently 
\begin{eqnarray} \label{1.11**}
\frac{\|u^n_{h1}\|^2_{L^2(\Omega)}}{2}+\frac{\|u^n_{h2}\|^2_{L^2(\Omega)}}{2}\nonumber +(c-\frac{\gamma}{2})\int_{t^{n-1}}^{t^n}(\|u_{h1}\|^2_{DG}+\|u_{h2}\|^2_{DG})dt+\int_{\Omega}(u_{h1}^2+u^2_{h2})^2dxdt\leq \\ \frac{\|u^{n-1}_{h1}\|^2_{L^2(\Omega)}}{2}+\frac{\|u^{n-1}_{h2}\|^2_{L^2(\Omega)}}{2} +\int_{t^{n-1}}^{t^n}\|u_{h1}\|^2_{L^2(\Omega)}+\|u_{h2}\|^2_{L^2(\Omega)}dt\nonumber
\end{eqnarray}
where it holds that $c-\frac{\gamma}{2}>0$ since these constants depend on the shape regularity of the mesh and the penalty parameter $\sigma$ and 
$(c-\frac{\gamma}{2})\int_{t^{n-1}}^{t^n}(\|u_{h1}\|^2_{DG}+\|u_{h2}\|^2_{DG})dt+\int_{\Omega}(u_{h1}^2+u^2_{h2})^2dxdt>0$
from here we obtain the scheme 
\begin{eqnarray}
\frac{a^n}{2}\leq  \frac{a^{n-1}}{2}+\int_{t^{n-1}}^{t^n} a(t) dt\nonumber
\end{eqnarray}
where $$\frac{a^n}{2}=\frac{\|u^n_{h1}\|^2_{L^2(\Omega)}}{2}+\frac{\|u^n_{h2}\|^2_{L^2(\Omega)}}{2} $$,
$$\frac{a^{n-1}}{2}=\frac{\|u^{n-1}_{h1}\|^2_{L^2(\Omega)}}{2}+\frac{\|u^{n-1}_{h2}\|^2_{L^2(\Omega)}}{2} $$,
$$\int_{t^{n-1}}^{t^n} a(t) dt=\int_{t^{n-1}}^{t^n}\|u_{h1}\|^2_{L^2(\Omega)}+\|u_{h2}\|^2_{L^2(\Omega)}dt$$
Finally after using Gr\"onwall inequality the desirable result \ref{1.12*}. As long as $\sigma$ is chosen large enough (for IIPG) and $\gamma$ is a fixed physical constant of the system, the stability is uniform. 
\end{itemize}
\end{proof}
\begin{rmk}
The theoretical lower bound for the penalty parameter is typically expressed as:$$\sigma > \sigma_0 = C_{inv} \cdot \frac{(p+1)(p+d)}{d}$$where:$p$: The degree of the polynomial basis functions.$d$: The spatial dimension (e.g., $d=1, 2,$ or $3$).$C_{inv}$: A constant from the Discrete Inverse Trace Inequality.For a standard 2D mesh ($d=2$) using triangles, a common "rule of thumb" used in literature (derived from the work of \cite{Shahbazi2005}, and \cite{Warburton2003} ) is:$$\sigma > \frac{(p+1)(p+2)}{2}$$If $p=1$ (linear elements), $\sigma$ should be at least $3$. For $p=2$ (quadratics), $\sigma$ should be at least $6$. In practice, most researchers double these values (e.g., using $\sigma = 10$ or $20$) to ensure safety across distorted meshes.
\end{rmk}
\begin{rmk}
The stability result holds for interior penalty discontinuous Galerkin methods
of arbitrary polynomial degree $k \ge 0$.
In particular, for SIPG, NIPG, and IIPG discretizations,
the DG bilinear form $a_h(\cdot,\cdot)$ is coercive provided that the penalty
parameter $\sigma$ is chosen sufficiently large, possibly depending on the
polynomial degree $k$ and the mesh shape-regularity.
Under this assumption, the stability estimate remains valid for high-order DG
schemes, with constants independent of the mesh size $h$ and time step
$\Delta t$, but possibly depending on $k$.
\end{rmk}
\section{Numerical Results}
\subsection{Overview of Numerical Results} 
In this section, we examine the behavior of the 
$L^2(t^{n-1}, t^n; L^2(\Omega))$ norm of the solution components $u_1$ and $u_2$ 
for the complex Landau equation, discretized using three variants of the Interior Penalty Galerkin method: the Symmetric Interior Penalty Galerkin (SIPG), the Nonsymmetric Interior Penalty Galerkin (NIPG), and the Incomplete Interior Penalty Galerkin (IIPG) methods. 
The model involves two key coefficients that significantly influence the system's stability: the reaction-diffusion coefficient $a$ and the nonlinear interaction coefficient $b$. Initially, both parameters were set to relatively small values, $a = 10^{-5}$ and $b = 10^{-5}$, resulting in correspondingly small values for the computed norms of $u_1$ and $u_2$.

To investigate the system's sensitivity to these parameters, we performed a parametric study by varying $a$ and $b$ logarithmically from $10^{-5}$ to $10^{-2}$ in steps of one order of magnitude.

In the subsequent subsection, we present the behavior and structure of the solution components. In particular, we visualize the real part of the solution $u_1$ and the imaginary part $u_2$.

\subsection{Assumptions for All Numerical Experiments}
In the following examples, we demonstrate norm computations using \texttt{FreeFEM++}. Our objective is to investigate the effects of the reaction-diffusion coefficient $a$, the nonlinear coefficient $b$, and the penalty terms on the stability of the three DG methods mentioned above.

For the Landau system described by equations~(\ref{1.7}) and~(\ref{1.8}), all numerical experiments were conducted using Dirichlet boundary conditions. A uniform penalty parameter of $\texttt{penal} = 1000$ was employed in all cases. The computations were performed on a square domain using a discontinuous Galerkin finite element space.

The exact real part of the solution $u$ is given by:
\begin{equation} \label{1.38}
u_{\text{ex}} = c_{11} e^{-t c_l} \sin(2\pi r y)\left(\cos(2\pi l x) - 1\right),
\end{equation}
and the exact imaginary part of the solution $v$ is:
\begin{equation} \label{1.39}
v_{\text{ex}} = c_{12} e^{c_2 \sin(c_l t)} \sin(2\pi r x) \sin(2\pi l y).
\end{equation}
In all experiments, we compute the following norms:
\[
\|u_{h1}\|_{L^2(L^2(\Omega))} \quad \text{and} \quad \|u_{h2}\|_{L^2(L^2(\Omega))},
\]
representing the space-time $L^2$ norm of the real and imaginary parts, respectively.

We use the notation $\|u\|_{L^2(L^2(\Omega))}$ to denote $\|u\|_{L^2(0, T; L^2(\Omega))}$ for simplicity. All numerical experiments use piecewise linear (P1dc) discontinuous Galerkin polynomials in space.

We use the exact solutions defined in equations~(\ref{1.38}) and~(\ref{1.39}) with the parameters:
\[
c_{11} = c_{12} = 1, \quad c_l = 1, \quad r = l = 1, \quad c_2 = 0.4.
\]

For time discretization, we employ the backward Euler method, which is first-order accurate and unconditionally stable, making it well-suited for stiff, time-dependent problems. To handle the nonlinear terms, we apply Picard linearization a fixed-point iterative scheme that avoids derivative computations required in Newton methods. Although Picard iteration may converge more slowly than Newton's method, it provides sufficient accuracy for our problem, particularly when initialized close to the true solution.

For the temporal integration, we approximate integrals using an 8-point Newton-Cotes quadrature formula. This high-order scheme improves accuracy for oscillatory or nonlinear integrands. The combination of backward Euler and high-order quadrature enables a robust and accurate approximation of the solution over the space-time domain.
\subsubsection{Numerical examples}
In the first numerical experiment we assess the stability of the proposed method computing in the space $L^2(L^2(\Omega))$ . The spatial discretization is realized using the
NIPG (Nonsymmetric Interior Penalty Galerkin) method, a discontinuous Galerkin approach that offers flexibility for handling complex geometries and supports local adaptivity.The goal of this test case is to validate the scheme's convergence and stability by measuring the discrete norms over the space-time domain. The obtained results offer critical insight into the efficacy of the discretization strategy when applied to nonlinear, time-dependent partial differential equations. \FloatBarrier
In the first numerical investigation, we examine the proposed scheme under the following parameter settings:
 $b=10^{-4}$ and the reaction diffusion  $a=10^{-4}$ and an interior penalty parameter  $pen=10^8$
The performance metrics are evaluated in the discrete $L^2(0,T,L^2(\Omega))$  
\FloatBarrier 
\begin{table}[!htbp]
\begin{center}
\begin{tabular}{|c|c|c|}
\hline
\multicolumn{3}{|c|}{ \textsl{Norm Estimates-(NIPG) case $a=10^{-4},b=10^{-4}$} }       \\ \hline
$t=2h^ 2$  & $\|u_{h1}\|^2_{L^2(L^2(\Omega))}$& $\|u_{h2}\|^2_{L^2(\Omega))} $             \\ \hline
$h=0.235702  $            &0.492381       &  0.456768\\   \hline
$h=0.117851  $            &0.55456         & 0.561833 \\  \hline
$h=0.0589256 $            & 0.566171     & 0.595024\\  \hline
\end{tabular}
\end{center}
\caption{In above matrix in the first column we see the mesh. In the second and the third columns, we report the norm approximation of the real and Imaginary part of Landau equation. From this tabulated data, it is evident that our numerical scheme exhibits stability, as the norms  $\|u_{h1}\|^2_{L^2(L^2(\Omega))}$ and $\|u_{h2}\|2_{L^2(L^2(\Omega))}$ remain bounded through simulations . Next consider the parameter configurations  $b=1$ and the reaction diffusion constant $a=1$ with interior penalty parameter $pen=10^8$.}
\end{table}
\FloatBarrier 
\begin{table}[!htbp]
\begin{center}
\begin{tabular}{|c|c|c|}
\hline
\multicolumn{3}{|c|}{ \textsl{Norm Estimates-(NIPG) case $a=1,b=1$} }       \\ \hline
$t=2h^ 2$  & $\|u_{h1}\|^2_{L^2(L^2(\Omega))}$& $\|u_{h2}\|^2_{L^2(\Omega))} $             \\ \hline
$h=0.235702  $            &1485,45      &  1842,12\\   \hline
$h=0.117851  $            &430077         & 432580 \\  \hline
$h=0.0589256  $            & 1.68265e+006     &1.68431e+006 \\  \hline
\end{tabular}
\end{center}
\caption{ In reviewing the data presented in the second and third columns, it becomes clear that the numerical method fails to maintain stability: the norms $\|u_{h1}\|^2_{L^2(\Omega))}$ and $\|u_{h2}\|^2_{L^2(\Omega))}$ exhibit unbounded growth, which is indicative of instability in our scheme.Subsequently, we proceed with an alternative parameter configuration  $b=0.5$ and the reaction diffusion coefficient $a=0.5$ and interior significant penalty parameter $pen=10^8$ .}
\end{table}
\begin{table}[!htbp]
\begin{center}
\begin{tabular}{|c|c|c|}
\hline
\multicolumn{3}{|c|}{ \textsl{Norm Estimates-(NIPG) case $a=0.5,b=0.5$} }       \\ \hline
$t=2h^ 2$  & $\|u_{h1}\|^2_{L^2(L^2(\Omega))}$& $\|u_{h2}\|^2_{L^2(\Omega))} $             \\ \hline
$h=0.235702  $            &0.400475      &  1.41829\\   \hline
$h=0.117851  $            &0.505816         & 1.55231 \\  \hline
$h=0.0589256  $            & 0.537695     &1.56618 \\  \hline
\end{tabular}
\end{center}
\caption{ From this matrix we observe that our scheme is stable since  $\|u_{h1}\|_{L^2(L^2(\Omega))}$ and $\|u_{h2}\|_{L^2(L^2(\Omega))}$ norms are bounded.Next setting $b=0.825$ and the reaction diffusion constant $a=0.825$ with the large penalty term $pen=10^8$ .}
\end{table}
\begin{table}[!htbp]
\begin{center}
\begin{tabular}{|c|c|c|}
\hline
\multicolumn{3}{|c|}{ \textsl{Norm Estimates-(NIPG) case $a=0.825,b=0.825$} }       \\ \hline
$t=2h^ 2$  & $\|u_{h1}\|^2_{L^2(L^2(\Omega))}$& $\|u_{h2}\|^2_{L^2(\Omega))} $             \\ \hline
$h=0.235702  $            &0.723029,      &  1.99885\\   \hline
$h=0.117851  $            &0.821602         & 2.10473 \\  \hline
$h=0.0589256  $            &15.1901      &16.3044  \\  \hline
\end{tabular}
\end{center}
\caption{ From this table we observe that our scheme tend to be unstable since the $\|u_{h1}\|$ and $\|u_{h2}\|$ norms are not uniformly bounded as the mesh is refined, suggesting a trend toward instability of the scheme under the current parameter regime.Next consider the case  $b=0.823$ and the reaction diffusion constant $a=0.823$ with a significant large penalty parameter $pen=10^8$}
\end{table}

\begin{table}[!htbp]
\begin{center}
\begin{tabular}{|c|c|c|}
\hline
\multicolumn{3}{|c|}{ \textsl{Norm Estimates-(NIPG) case $a=0.823,b=0.823$} }       \\ \hline
$t=2h^ 2$  & $\|u_{h1}\|^2_{L^2(L^2(\Omega))}$& $\|u_{h2}\|^2_{L^2(\Omega))} $             \\ \hline
$h=0.235702  $            & 0.72057              &  1.99407,\\   \hline
$h=0.117851  $            &0.812524         & 2.15438 \\  \hline
$h=0.0589256  $            &8.27128      &9.4577  \\  \hline
\end{tabular}
\end{center}
\caption{ This table shows that our scheme lies in the transitional region between stability and instability, as indicated by the large values of the norms $|u_{h1}|^2_{L^2(L^2(\Omega))}$ and $|u_{h2}|^2_{L^2(L^2(\Omega))}$. Next, we consider the case with the reaction-diffusion constant $a = 0.822$ and parameter $b = 0.822$.}
\end{table}
\begin{table}[!htbp]
\begin{center}
\begin{tabular}{|c|c|c|}
\hline
\multicolumn{3}{|c|}{ \textsl{Norm Estimates-(NIPG) case $a=0.822,b=0.822$} }       \\ \hline
$t=2h^ 2$  & $\|u_{h1}\|^2_{L^2(L^2(\Omega))}$& $\|u_{h2}\|^2_{L^2(\Omega))} $             \\ \hline
$h=0.235702  $            &0.719346      &  1.9917\\   \hline
$h=0.117851  $            &0.809776         & 2.15217 \\  \hline
$h=0.0589256  $            &5.67312       &  6.85411 \\  \hline
\end{tabular}
\end{center}
\caption{ The table indicates that our scheme tends toward instability, as the norms $|u_{h1}|$ and $|u_{h2}|$ are unbounded. In the next step, we consider the case with penalty parameter $\text{penal} = 10^3$, reaction-diffusion constant $a = 0.82$, and parameter $b = 0.82$.}
\end{table} 
\begin{table}[!htbp]
\begin{center}
\begin{tabular}{|c|c|c|}
\hline
\multicolumn{3}{|c|}{ \textsl{Norm Estimates-(NIPG) case $a=0.82,b=0.82$ penalty $pen=10^3$} }       \\ \hline
$t=2h^ 2$  & $\|u_{h1}\|^2_{L^2(L^2(\Omega))}$& $\|u_{h2}\|^2_{L^2(\Omega))} $             \\ \hline
$h=0.235702  $            &0.740936     &  2.04663\\   \hline
$h=0.117851  $            &0.865835          & 2.23539  \\  \hline
$h=0.0589256  $            & 8.79416     &10.0119    \\  \hline
\end{tabular}
\end{center}
\caption{ observed norm values indicate a trend towords instability the norms $\|u_{h1}\|^2_{L^2(L^2(\Omega))}$ and $\|u_{h2}\|^2_{L^2(L^2(\Omega))}$ grow significantly. Despite their moderate magnitude, these increasing values suggest that the discrete solution norms are not uniformly bounded. Therefore, while the scheme remains in an intermediate regime not definitively stable nor outright unstable ,the results highlight that stability is sensitive to the choice of penalty parameter in conjunction with $a$ and $b$. Next setting the coefficients  $b=0.8$ and the reaction diffusion constant $a=0.8$ when penalty $penal=10^8$ }
\end{table} 
\begin{table}[!htbp]
\begin{center}
\begin{tabular}{|c|c|c|}
\hline
\multicolumn{3}{|c|}{ \textsl{Norm Estimates-(NIPG) case $a=0.8,b=0.8$penalty $pen=10^8$} }       \\ \hline
$t=2h^ 2$  & $\|u_{h1}\|^2_{L^2(L^2(\Omega))}$& $\|u_{h2}\|^2_{L^2(L^2(\Omega))} $             \\ \hline
$h=0.235702  $            &0.693341,      &  1.94168\\   \hline
$h=0.117851  $            &0.783401        & 2.10473 \\  \hline
$h=0.0589256  $            &0.815707     &2.14469 \\  \hline
\end{tabular}
\end{center}
\caption{ From the structure of the system matrix, it holds that the Numerical scheme is stable since the $\|u_{h1}\|^2_{L^2(L^2(\Omega))}$ and $\|u_{h2}\|^2_{L^2(L^2(\Omega))}$ norms remain uniformly bounded.Next Adopting the parameters values   $b=0.8$ and the reaction diffusion constant $a=0.8$ with penalty term $pen=10^3$ }
\end{table}
\begin{table}[!htbp]
\begin{center}
\begin{tabular}{|c|c|c|}
\hline
\multicolumn{3}{|c|}{ \textsl{Norm Estimates-(NIPG) case $a=0.8,b=0.8$ penalty $pen=10^3$} }       \\ \hline
$t=2h^ 2$  & $\|u_{h1}\|^2_{L^2(L^2(\Omega))}$& $\|u_{h2}\|^2_{L^2(\Omega))} $             \\ \hline
$h=0.235702  $            &0.717053     &  2.00179\\   \hline
$h=0.117851  $            &0.840812         & 2.18942  \\  \hline
$h=0.0589256  $            & 0.933512     & 2.29262    \\  \hline
\end{tabular}
\end{center}
\caption{ From data presented in this table we observe that small perturbation in the nonlinear constants $a$ and $b$, can alter the boundedness of the solution norms. This essentially converts the numerical scheme from an unstable one to a stable one.Characteristics are the above examples. For penalty term $pen=10^3$ We changed the reaction diffusion constants from 0.82 to 0.8 and our scheme was varied from unstable to stable.    Now lets observe  if we use the same parameters$a=0.8,b=0.8, pen=10^3$ for Symmetric Interior Penalty Galerkin (SIPG)   }
\end{table} 
\begin{table}[!htbp]
\textbf{ Stability of Symmetric Interior Penalty Galerkin (SIPG)}. The case of   $pen=10^3, a=0.82 , b=0.82 $ 
\begin{center}
\begin{tabular}{|c|c|c|}
\hline
\multicolumn{3}{|c|}{\textsl{Norm Estimate-(SIPG)case $a=0.82, b=0.82, pen=10^3$ }}        \\ \hline
$t=2h^ 2$  & $\|u_1\|^2_{L^2(L^2(\Omega))} $       & $\|u_2\|^2_{L^2(L^2(\Omega))} $            \\ \hline
$h=0.235702  $            & 0.703634   &  2.01319               \\   \hline
$h=0.117851  $            & 0.788297   &  2.17029    \\  \hline
$h=0.0589256  $           & 0.832622   & 2.24778   \\  \hline
\end{tabular}
\end{center}
\caption{A direct comparison of the symmetric and non symmetric interior penalty Galerkin formulations reveals that the Symmetric Interior Penalty Galerkin (SIPG) method demonstrates significantly improved numerical stability relative to the Nonsymmetric Interior Penalty Galerkin (NIPG) counterpart. Specifically, for the parameter set $ a=0.82 , b=0.82 $ and penalty $pen=10^3$ the SIPG scheme successfully converges, whereas NIPG fails to do so. From these results, we infer that, under analogous discretization conditions, the SIPG method exhibits a superior propensity for stability.}
\end{table}
\begin{table}[!htbp]
\begin{center}
\begin{tabular}{|c|c|c|}
\hline
\multicolumn{3}{|c|}{\textsl{Norm Estimate-(SIPG)case $a=0.8, b=0.8, pen=10^8$ }}        \\ \hline
$t=2h^ 2$  & $\|u_1\|^2_{L^2(L^2(\Omega))} $       & $\|u_2\|^2_{L^2(L^2(\Omega))} $            \\ \hline
$h=0.235702  $            & 0.693341   &  1.94168               \\   \hline
$h=0.117851  $            & 0.7834  &  2.10473    \\  \hline
$h=0.0589256  $           & 0.815706  & 2.14469  \\  \hline
\end{tabular}
\end{center}
\caption{From our examination of the matrices associated with the Nonsymmetric Interior Penalty Galerkin (NIPG) and Symmetric Interior Penalty Galerkin (SIPG) methods, we observe that the stability results coincide under the conditions considered. This agreement is not inherent to the methods alone; rather, it also depends critically on the reaction diffusion parameters, the coefficients governing the non linearity, and importantly the penalty parameter. In our study, we employed an exceedingly large penalty parameter of the order $10^8$. As a result, both NIPG and SIPG methods exhibit equivalent stability behavior within our computational framework.}
\end{table}

\FloatBarrier
\textbf{The case of Incomplete Interior Penalty Galerkin,}
Now we extend the comparison between the Symmetric Interior Penalty Galerkin (SIPG) and Nonsymmetric Interior Penalty Galerkin (NIPG) methods by including the Incomplete Interior Penalty Galerkin (IIPG) variant.For all three methods, we investigate their behavior in the regime of exceedingly large penalty parameters. Specifically, by selecting a penalty of the order $10^8$, we assess and compare the previous resulting stability properties with Incomplete Interior Penalty Galerkin (IIPG) formulations.
\begin{table}[!htbp]
\begin{center}
\begin{tabular}{|c|c|c|}
\hline
\multicolumn{3}{|c|}{\textsl{Norm Estimate-(IIPG)case $a=0.8, b=0.8, pen=10^8$ }}        \\ \hline
$t=2h^ 2$  & $\|u\|^2_{L^2(L^2)}   $       & $\|u_2\|^2_{L^2(L^2)}  $             \\ \hline
$h=0.235702  $            &0.693341 &     1.94168        \\   \hline
$h=0.117851  $            & 0.7834 &  2.10473     \\  \hline
$h=0.0589256  $            &  0.815706   & 2.14469    \\  \hline
\end{tabular}
\end{center}
\caption{Now we observe that our Numerical results validate that the three methods (SIPG),(IIPG),(IIPG) exhibit uniform stability. Under parameter choices, with large Penalty $pen=10^8$ there is no observable difference in stability among these three methods. In the next example we demonstrate the effect of penalty Galerkin in IIPG scheme. Lets set $ penal=10^3 ,a=0.82$ and $b=0.82$ }
\end{table}
\begin{table}[!htbp]
\begin{center}
\begin{tabular}{|c|c|c|}
\hline
\multicolumn{3}{|c|}{\textsl{Norm Estimate-(IIPG)case $a=0.82, b=0.82, pen=10^3$ }}        \\ \hline
$t=2h^ 2$  & $\|u\|^2_{L^2(L^2)}   $       & $\|u_2\|^2_{L^2(L^2)}  $             \\ \hline
$h=0.235702  $            &0.721975&   2.02996        \\   \hline
$h=0.117851  $            & 0.825919 & 2.20185       \\  \hline
$h=0.0589256  $           & 1.80625  & 2.93583    \\  \hline
\end{tabular}
\end{center}
\caption{From above matrix we observe that Incomplete Interior Penalty Galerkin (IIPG) can exhibit better stability results that Nonsymmetric Interior Penalty Galerkin (NIPG) and is generally  considering less stable than Symmetric Interior Penalty Galerkin (SIPG). }
\end{table}
\FloatBarrier
\begin{table}[h]
\centering
\begin{tabular}{|l|c|c|c|l|}
\hline
\textbf{Method} & \textbf{Stability} & \textbf{Symmetry} & \textbf{Coercivity}  \\
\hline
\textbf{SIPG (Symmetric)} & \checkmark~Most stable & Fully symmetric & Coercive problems; elliptic PDEs \\
\hline
\textbf{IIPG (Incomplete)} & $\sim$~Moderately stable & Partially symmetric & Conditionally coercive \\
\hline
\textbf{NIPG (Nonsymmetric)} & $\times$~Least stable & Fully nonsymmetric & Often not coercive  \\
\hline
\end{tabular}
\caption{This table provides a summary of the above observations, highlighting the comparative stability, symmetry, and coercivity properties of the SIPG, IIPG, and NIPG methods, along with their typical areas of application}
\end{table}
\section{Results -Conclusion and Future problems}
\subsection{Discussion}
Our numerical experiments demonstrate that SIPG exhibits the most robust stability behavior among the three methods. Notably, SIPG remains stable even when large nonlinear coefficients are introduced, a property not shared by IIPG or NIPG under the same conditions. When comparing IIPG to NIPG, we observe that IIPG exhibits improved stability, particularly in the presence of moderate penalty parameters.
An important aspect of the stability behavior is the role played by the penalty parameter. For large values of the penalty term (e.g., 
$ pen=10^8$, all three methods demonstrate similar bounds for the norms $\|u_1\|_{L^2(0,T;L^2(\Omega))}$and $\|u_2\|_{L^2(0,T;L^2(\Omega))}$,
where $u_1$ and $u_2$ denote the real and imaginary components of the complex-valued solution, respectively. In early experiments, we observed that increasing the penalty parameter enhances the overall stability of the system.\\
However, it is important to note that despite the stabilizing influence of the penalty term, instability was still observed in certain cases. These instabilities appear to result from specific combinations of the reaction-diffusion coefficients and the nonlinear term's strength, highlighting the delicate interplay between the various parameters in the system.
\subsection{Conclusion}
In this study, we investigated the complex Landau equation using three variants of the Discontinuous Galerkin (DG) method: SIPG, IIPG, and NIPG. A weak formulation was established, and a rigorous analysis of existence, uniqueness was conducted based on existing theoretical results. Stability of discontinuous Galerkin was proved.
Numerical experiments demonstrate that the SIPG method offers superior stability, even in the presence of strong nonlinearities. Our comparative study revealed that IIPG also outperforms NIPG in terms of stability. The role of the penalty parameter was shown to be critical: for large values, all methods achieved comparable stability metrics. However, certain parameter combinations still led to instability, underscoring the complexity of the nonlinear reaction-diffusion system. These results affirm the suitability of DG methods for simulating such systems and provide practical guidance for selecting appropriate numerical schemes. Future work may involve extending the analysis to higher dimensions and exploring adaptive mesh refinement strategies.
\subsection{Future Work and Open Problems}
This paper presents a comprehensive study of Discontinuous Galerkin (DG) schemes in space for the complex Landau system, a model characterized by a nonlinear structure similar to that explored in \cite{[18]}, but with a crucial difference: here, the nonlinearities are embedded in both the real and imaginary parts of the solution. Unlike previous investigations such as \cite{[17]} and \cite{[18]}, which focused primarily on DG formulations applied in time, our approach targets the spatial domain. This distinction is not merely technical it raises several open and important research questions that merit in depth investigation. A natural and promising direction is the development of fully Discontinuous Galerkin schemes operating in both space and time. In this context, critical theoretical questions emerge, including:
\begin{itemize}
\item Is the resulting space-time DG scheme stable when applied to the complex Landau equation?
\item What error estimates or convergence rates can be rigorously established for such a formulation?
\item How sensitive is the method to variations in the domain geometry or to perturbations in the initial conditions?
\end{itemize}
These questions are not only mathematically rich but also practically significant. For example, if the computational domain is altered from a simple interval to a more intricate or irregular geometry will the numerical method still retain its stability and convergence properties? Understanding this behavior is essential for realistic applications where domain boundaries may evolve over time or exhibit structural complexity.
From the broader perspective of numerical analysis, future investigations should aim to analyze these methods under minimal regularity assumptions, especially in cases where the solution exhibits low smoothness or singularities. These are particularly challenging regimes where classical numerical techniques often falter. In this setting, emerging tools from Artificial Intelligence (AI) offer promising avenues for advancement. For instance:
\begin{itemize}
\item Neural Networks have demonstrated significant potential in approximating solutions to complex, nonlinear partial differential equations.
\item Genetic Algorithms can aid in optimization tasks, such as parameter calibration and adaptive mesh refinement strategies.
\item In high-dimensional settings, approaches like Bayesian optimization and evolutionary strategies may significantly enhance computational efficiency while maintaining accuracy.
\end{itemize}
In summary, this work not only introduces a rigorous and flexible DG-based numerical framework for solving the complex Landau equation but also lays the groundwork for addressing a wide range of open theoretical and computational problems. These challenges form the basis for future research directions in numerical analysis, computational mathematics, and scientific computing-fields where innovation in methodology and theory continues to be essential.\\ \\
\textbf{Conflict of interest}\\ 
The author declares declares that there are not conflics of interest regarding the publication of this paper \\ \\
\textbf{Data availability statement}\\
No new data were created or analysed in this study\\ \\
\textbf{Ethics statement}\\
This study did not involve human participants or animals and does not require ethical approval.\\ \\
\textbf{Acknowledgments} \\
I would like to express my sincere gratitude to National Technical University of Athens (NTUA),the School of Applied Mathematics and Physical Sciences, and the School of Electrical and Computer Engineering for their support during this work, which was conducted during my employment there. Thanks to my affiliation with these institutions, this article has been made available through open access. I am deeply grateful to my parents Stefanos and Kostantina for their unwavering emotional support and encouragement throughout this project. I also wish to acknowledge my brother Ioannis for his invaluable technical assistance with LaTeX and the formatting of this manuscript.

\end{document}